\documentclass[11pt,reqno]{amsart}
\usepackage{amssymb,amsmath}
\usepackage{kpfonts}
\usepackage{subfig}\newcommand{\be}{\begin{eqnarray}}
\newcommand{\ee}{\end{eqnarray}}
\newcommand{\beno}{\begin{eqnarray*}}
\newcommand{\eeno}{\end{eqnarray*}}
\newcommand{\clos}{\mbox{\rm clos}}\newcommand{\e}{{\varepsilon}}\newcommand{\R}{{\mathbb R}}\newcommand{\Z}{{\mathbb Z}}

\newcommand{\M}{{\mathcal M}}

\newcommand{\om}{\omega}
\newtheorem{theorem}{Theorem}\newtheorem{lemma}[theorem]{Lemma}\newtheorem{cor}[theorem]{Corollary}\theoremstyle{definition}
\newtheorem{defi}[theorem]{Definition}
\newtheorem{assumption}[theorem]{Assumption}
\theoremstyle{remark}\newtheorem{remark}[theorem]{Remark}\numberwithin{equation}{section}\input epsf.sty
\begin{document}\thispagestyle{empty}

%
%
%
%
%
%
%
%
%
%
%
%
\newcommand{\Vol}{Vol}
\newcommand{\s}[0]{\theta}
\newcommand{\T}{{\mathbb T}}
\newcommand{\G}{{\mathcal G}}
\newcommand{\C}{{\mathbb C}}
\newcommand{\Pa}{{P_{1,\theta}(y)}}
\newcommand{\Pb}{{P_{2,\theta}(y)}}
\newcommand{\Pshp}{{P_{1,t}^\sharp(x)}}
\newcommand{\Pflt}{{P_{1,t}^\flat(x)}}
\newcommand{\OTL}{{\Omega(\theta,\ell)}}
\newcommand{\rsz}{{R(\theta^*)}}
\newcommand{\phitil}{{\tilde{\varphi}_t}}
\newcommand{\lambdatil}{{\tilde{\lambda}}}
\newcommand{\al}{\alpha}
\newcommand{\ftpI}{\langle (g^2 + \tau^2 f^2)^{\frac{p}{2}} \rangle_{I}}
\newcommand{\fI}{\langle f \rangle_{I}}
\newcommand{\fpI}{\langle |f|^p \rangle_{I}}
\newcommand{\gI}{\langle g \rangle_{I}}
\newcommand{\mtp}{|(f,H_J)| = |(g,H_J)|}
\newcommand{\fJp}{\langle f \rangle_{J^+}}
\newcommand{\fJm}{\langle f \rangle_{J^-}}
\newcommand{\gJp}{\langle f \rangle_{J^+}}
\newcommand{\gJm}{\langle f \rangle_{J^-}}
\newcommand{\fIpm}{\langle f \rangle_{J^{\pm}}}
\newcommand{\gIpm}{\langle g \rangle_{J^{\pm}}}
\newcommand{\fIp}{\langle f \rangle_{I^{+}}}
\newcommand{\fIm}{\langle f \rangle_{I^{-}}}
\newcommand{\gIp}{\langle g \rangle_{I^{+}}}
\newcommand{\gIm}{\langle g \rangle_{I^{-}}}
\newcommand{\yone}{y_1}
\newcommand{\ytwo}{y_2}
\newcommand{\yonep}{y_1^+}
\newcommand{\ytwop}{y_2^+}
\newcommand{\yonem}{y_1^-}
\newcommand{\ytwom}{y_2^-}
\newcommand{\ythr}{y_3}

\newcommand{\xone}{x_1}
\newcommand{\xtwo}{x_2}
\newcommand{\xonep}{x_1^+}
\newcommand{\xtwop}{x_2^+}
\newcommand{\xonem}{x_1^-}
\newcommand{\xtwom}{x_2^-}
\newcommand{\xthr}{x_3}
\newcommand{\ptwo}{\frac{p}{2}}
\newcommand{\alp}{\al^+}
\newcommand{\alm}{\al^-}
\newcommand{\alpm}{\al^{\pm}}
\newcommand{\twrp}{\frac{2}{p}}
\newcommand{\sbeta}{\sqrt{\om^2 - \tau^2}}
\newcommand{\ssbeta}{\om^2 - \tau^2}
\newcommand{\sign}{\operatorname{sign}}
\newcommand{\rst}[1]{\ensuremath{{\mathbin\upharpoonright}%

\newcommand{\N}[0]{{\mathbb N}}
\newcommand{\Z}[0]{{\mathbb Z}}
\newcommand{\M}[0]{{\mathbb M}}
\newcommand{\cK}[0]{{\mathcal K}}
\newcommand{\f}[0]{\varphi}
\newcommand{\pd}[0]{\partial}

\raise-.5ex\hbox{$#1$}}}

\title[Lower bounds for vector-valued and matrix-valued multipliers in $L^p$]{{Sharp Lower bound estimates for vector-valued and matrix-valued multipliers in $L^p$}}
\author{Nicholas Boros}\address{Nicholas Boros, Dept. of Math., Michigan State University.
{\tt borosnic@msu.edu}}
\author{Alexander Volberg}\address{Alexander Volberg, Dept. of  Math., Michigan State University.
{\tt volberg@math.msu.edu}}

\begin{abstract}We generalize the idea of a multiplier in two different ways.  First of all, we consider multipliers in the form of a vector acting on a scalar function.  Using this technique we are able to show that for $\tau^2 \leq p^*-1$ and $1<p<2$ or $\tau \in \R$ and $2 \leq p < \infty$, $\left\|\left(\begin{array}{c}
ReB\\
\tau I \end{array} \right)\right\|_{L^p(\C,\C) \to L^p(\C,\C^2)} \geq ((p^*-1)^2+\tau^2)^{\frac 12},$ where $B$ is the Ahlfors--Beurling operator, $I$ is the identity operator and $p^*-1 = \max\left\{p-1, \frac{1}{p-1}\right\}.$  Secondly, we consider matrix--valued multipliers to obtain a new proof showing that $\|B\|_{L^p(\C) \to L^p(\C)} \geq p^*-1.$
\end{abstract}
\maketitle

\section{\bf Introduction}

The results of Geiss, Montgomery-Smith, Saksman \cite{Saks} generalized from an idea of Bourgain \cite{Bo} seems to be quite robust.  Bourgain originally used this technique for the Hilbert transform, to get an $L^p$ lower bound in terms of $UMD$ constant.  Geiss, Montgomery-Smith, Saksman were able to generalize this quite a bit further to the class of real-valued, even multipliers that are continuous and homogoenous of order zero.  The argument, with a little bit of effort will generalize even further.  We are able to get a lower bound of an operator corresponding to either vector-valued multipliers or matrix-valued multipliers also in terms of the natural generalization of $UMD.$  With matrix valued multipliers, one can get the lower bound of an $L^p$ operator corresponding to a complex multiplier in terms of $UMD.$  We present the Ahlfors--Beurling operator as such an example. 

\section{\bf Vector--valued Multipliers}

\subsection{\bf Definitions and Notation}
Let $1 \leq p,p_0 \leq \infty$ and $m:\R^n \to \C$ be bounded function.  We say that $m$ is the multiplier corresponding with the operator $T_m$, if $T_mf:= (m \widehat{f})^{\vee}$ is a bounded operator from $L^p(\R^n) \to L^{p_0}(\R^n)$, considering functions $f \in \mathcal{S}$. Where $\,\,\widehat{} \,\,$ and $^\vee$ are the Fourier transform and its inverse and $\mathcal{S}$ is the Schwartz class.  The space of all such $m$ is denoted $\mathscr{M}_{p,p_0}(\R^n)$ with norm $\|m\|_{\mathscr{M}_{p,p_0}(\R^n)} := \|T_m\|_{L^p \to L^{p_0}}.$

Similarly, for $1 \leq p,p_0 \leq \infty$ we can consider a bounded sequence $a = \{a_k\}_{k \in \Z^n} \subset \C$.  We say that $a$ is the discrete multiplier corresponding the operator $T_a$, if 
$$(T_a f)(\theta):= (a \widehat{f})^{\vee}(\theta) = \sum_{k \in \mathbf{Z^n}}{a_k \widehat{f}(k)e^{i(k,\theta)}}$$ 
is a bounded operator from $L^p(\T^n) \to L^{p_0}(\T^n)$, considering functions $f$ being trigomometric polynomials on $\T^n := [-\pi, \pi)^n$.  The space of all such $a$ is denoted $\mathscr{M}_{p,p_0}(\Z^n)$ with norm $\|a\|_{\mathscr{M}_{p,p_0}(\Z^n)} := \|T_m\|_{L^p \to L^{p_0}}.$  In this paper, we will just take $a_k = m(k)$ and call $\widetilde{m}$ the discrete multiplier under this restriction so that the operator is defined as 
$(T_{\widetilde{m}}f)(\theta)= \sum_{k \in \mathbf{Z^n}}{m(k) \widehat{f}(k)e^{i(k,\theta)}}$.
 
Generalizing the idea of a muliplier, let us consider $M = \{m_{i}\}_{1 \leq i \leq m}$ as a vector-valued function whose entries are bounded functions from $\R^d$ to $\C$ and $p,p_0$ such that $1 \leq p,p_0 \leq \infty$.  We say that $M$ is the multiplier corresponding to the operator $T_M$, if $T_Mf:= (M \widehat{f})^{\vee}$ (the inverse Fourier transform defined componentwise) is a bounded operator from $L^p(\R^d,\C)$ to $L^{p_0}(\R^d,\C^m)$, considering functions $f$ from $\R^d$ to $\C$ in $\mathcal{S}$.  Note that $\widehat{T_Mf} = \widehat{f}M$ as an extension of the scalar case and $\|f\|_{L^{p_0}(\R^d,\C^m)}^{p_0} := \int_{\R^d}{\|f(x)\|_{\C^m}^{p_0} \, dx}$.  We use the notation $\widetilde{M} = (\widetilde{m}_{i})_{1 \leq i \leq m}$ to denote the discrete multiplier.  

\begin{remark}
We will use the notation $\mathscr{V}_{m}^{p,p_0}(\R^d)$ to denote the class of all continuous (except possibly at zero) and homogeneous of order zero vector multipliers, where $\|M\|_{\mathscr{V}_{m}^{p,p_0}} := \|T_M\|_{L^p(\R^d, \C) \to L^{p_0}(\R^d, \C^m)}$ for any such $M \in \mathscr{V}_{m}^{p,p_0}(\R^d).$
\end{remark}

\subsection{\bf Operator Norm of a Quadratic Perturbation of Martingale Transform}
\label{}

Let $\{r_n\}_{n \geq 0}$ the Rademacher sequence.  Let $\{F_N\}_{N \geq 0}$ and $\{G_N\}_{N \geq 0}$ be $\C$-valued martingale difference sequences of the form 
$$F_N = \sum_{k=1}^N{d_k(r_0,\dots,r_{k-1})r_k},\,\,\,\, G_N = \sum_{k=1}^N{\beta_k d_k(r_0,\dots,r_{k-1})r_k}$$
where $d_{k}:\{\pm 1\}^k \to \C$ is integrable and $\vec{\beta}$ is a vector with entries $\beta_{i} \in {\pm 1}$.  For any particular such $\vec{\beta}$, $MT_{\vec{\beta}}(F_N) := G_N$ is a martingale transform.  Let $\tau \in \R$ and $\left(\begin{array}{c}
MT_{\vec{\beta}}\\
\tau I \end{array} \right): L^p(\C) \to L^{p_0}(\C^2)$ be defined as $F_N \mapsto \left(\begin{array}{c}
G_N\\
\tau F_N \end{array} \right).$  Then 

\be \label{UMDvectordef}
\left\|\left(\begin{array}{c}
MT_{\vec{\beta}}\\
\tau I \end{array} \right)\right\|_{L^p(\C) \to L^{p_0}(\C^2)} = \sup_{\vec{\beta}}\frac{\left\|\sum_{k=1}^N{\left(\begin{array}{c}
\beta_k\\
\tau \end{array}\right) d_k(r_0,\dots,r_{k-1})r_k}\right\|_{p_0}}{\|\sum_{k=1}^N{d_k(r_0,\dots,r_{k-1})r_k}\|_p}.
\ee

\begin{defi}
\label{UMDpp0}
$C_{p,p_0}^{\tau} := \left\|\left(\begin{array}{c}
MT_{\vec{\beta}}\\
\tau I \end{array} \right)\right\|_{L^p(\C) \to L^{p_0}(\C^2)}.$
\end{defi}

Note that $C_{p,p}^{\tau} = ((p^*-1)^2+\tau^2)^{\frac 12}$ for $\tau^2 \leq p^*-1$ and $1<p<2$ or $\tau \in \R$ and $2 \leq p < \infty$ and shown in \cite{BJV1}, \cite{BJV2} and \cite{BJV3},where $p^*=\max\{p,\frac{p}{p-1}\}$.  

\subsection{\bf Transference}
\label{transference}

The following results are called transference properties of multipliers.  The main result that we will prove is a lower bound for $\|T_M\|_{L^p(\R) \to L^{p_0}(\R^d)}.$  Transference allows us to reduce this to getting a lower bound for $\|T_M\|_{L^p(\T) \to L^{p_0}(\T^d)},$ which is easier to work with. 

\begin{lemma}
\label{transference1}
Suppose $1 < p_0 \leq p < \infty$ and $M \in \mathscr{V}_{m}^{p,p_0}(\R^d),$ then $\|T_{M}\|_{L^{p}(\R^d,\C) \to L^{p_0}(\R^d,\C^m)} \geq \|T_{\widetilde{M}}\|_{L^{p}(\T^d,\C) \to L^{p_0}(\T^d,\C^m)}$.
\end{lemma}

\begin{proof}
The proof that we present here is a generalization of the proof in \cite{Graf}, on pp. 221--223, for $m$ being a scalar valued multiplier, so we will briefly describe the difference in the argument needed.  We claim that for $P$ and $Q$ being trigonometric polynomials, $L_{\e}(x) = e^{-\pi \e |x|^2}$ and $\frac{1}{p_0} + \frac{1}{q_0}=1$ we have
$$\int_{\T^d}{((T_{\widetilde{M}}P)(x),Q(x))dx}\\ = \lim_{\e \to 0} \e^{\frac{n}{2}} \int_{\R^n}{(T_{M}(PL_{\frac{\e}{p_0}})(x), Q(x)L_{\frac{\e}{q_0}}(x))dx}.$$
Indeed, by linearity we can just consider $P(x) = e^{2\pi i (j,x)}$ and $Q(x) = e^{2\pi i (k,x)}b,$ where $b \in \C^m.$  By Parseval's identity we have that $\int_{\T^d}{(T_{\widetilde{M}}P(x), Q(x))_{\C^m}dx} = (M(j), b)_{\C^m}\delta_{jk}.$  On the other hand 
\beno
&\,\,& \e^{\frac d2} \int_{\R^d}{(T_M(PL_{\frac \e p_0})(x), Q(x)L_{\frac \e q_0}(x))_{\C^m}dx}\\ 
&=& \left(\frac{\e}{p_0 q_0}\right)^{-\frac d2}\int_{\R^d}{(M(\xi), b)_{\C^m} e^{-\frac{-p_0 \pi |\xi -j|^2}{\e}}e^{-\frac{-q_0 \pi |\xi -k|^2}{\e}}d\xi} =: I 
\eeno
We get that $I \to \delta_{jk}$ as $\e \to 0$ proving the claim.

Now we consider $P$ and $Q$ as trigonometric polynomials, not just the particular monomials used in the proof of the claim.  Let $\frac{1}{p} + \frac{1}{q}=1$ and now we can estimate and get

\beno
\left|\int_{\T^n}{((T_{\widetilde{M}}P)(x),Q(x))dx}\right|\\
&=& \left|\lim_{\e \to 0} \e^{\frac{d}{2}} \int_{\R^d}{(T_{M}(PL_{\frac{\e}{p_0}})(x), Q(x))_{\C^m}L_{\frac{\e}{q_0}}(x)dx}\right|\\ 
& \leq &  \limsup_{\e \to 0}\e^{\frac{d}{2}}\left\|T_{M}(PL_{\frac{\e}{p_0}})\right\|_{L^{p_0}(\R^d)} \left\|QL_{\frac{\e}{q_0}}\right\|_{L^{q_0}(\R^d)}\\ 
& \leq & \limsup_{\e \to 0} \|T_M\|\left\|\e^{\frac{d}{2p_0}}PL_{\frac{\e}{p_0}}\right\|_{L^{p}(\R^d)} \left\|\e^{\frac{d}{2q_0}}QL_{\frac{\e}{q_0}}\right\|_{L^{q_0}(\R^d)}\\ 
&=& \|T_M\|\left\|P\right\|_{L^{p}(\T^d)} \left\|Q\right\|_{L^{q_0}(\T^d)},
\eeno
The last line uses the fact that for a continuous function $g$ the Poisson summation formula gives
\beno
\lim_{\e \to 0} \e^{\frac{dp}{2p_0}}\int_{\R^d}{g(x)e^{-\frac{\e p\pi |x|^2}{p_0}}dx} = \int_{\T^d}{g(x)dx}
\eeno
and similarly with $q, q_0$ replacing $p, p_0.$  Note that $\|T_M\| = \|T_M\|_{L^{p}(\R^d,\C) \to L^{p_0}(\R^d,\C^m)}.$
\end{proof}

\begin{lemma}
\label{transference2}
Suppose $1 < p \leq p_0 < \infty$ and $M \in \mathscr{V}_{m}^{p,p_0}(\R^d),$ then $\|T_{M}\|_{L^{p}(\R^d,\C) \to L^{p_0}(\R^d,\C^m)} \leq \|T_{\widetilde{M}}\|_{L^{p}(\T^d,\C) \to L^{p_0}(\T^d,\C^m)}$.
\end{lemma}

The proof of this is a similar modification to the scalar case (see pp. 223 -- 225 in \cite{Graf}) as was just shown in Lemma \ref{transference1}.

\subsection{\bf Lower Bound of Vector-Valued Multipliers} \label{lowerbound}

Let $Q = \T^d$ and $\mathscr{P}_k$ be the trigonometric polynomials from $Q^k$ to $\C$.  The following two theorems are generalizations of the results in \cite{Saks}.  We give a slightly different argument to both that also require some modifications in the more general setting.

\begin{theorem}
\label{bourgain}
Let $M \in \mathscr{V}_{m}^{p, p_0}(\R^d)$ be the multiplier corresponding to $T_M$ and 
$$T_{\widetilde{M}}^{k}\left(\sum_{j_1 \in \Z^n} \dots \sum_{j_k \in \Z^n}{x_{j_1, \dots, j_k}e^{i (j_1,\theta_{1})} \dots e^{i (j_k,\theta_{k})}}\right) :=
\sum_{j_1 \in \Z^n} \dots \sum_{j_k \in \Z^n}{M(j_k)x_{j_1, \dots, j_k}e^{i (j_1,\theta_{1})} \dots e^{i (j_k,\theta_{k})}}$$
If $\Phi_{k} \in E_k = \clos_{L^p}\{f \in \mathscr{P}_k: \int_{Q}{f(\theta_{1}, \dots, \theta_{k}) d\theta_{k}}=0\}$ and $1<p_0\leq p<\infty,$ then 
$$\left\|\sum_{k=1}^{N}{T_{\widetilde{M}}^{k}\Phi_k(\theta_{1}, \dots, \theta_{k})}\right\|_{L^{p_0}(Q^N)} \leq \|T_{\widetilde{M}}\|_{L^{p}(Q,\C) \to L^{p_0}(Q,\C^m)} \left\|\sum_{k=1}^{N}{\Phi_k(\theta_{1}, \dots, \theta_{k})}\right\|_{L^{p}(Q^N)}$$
for all $\Phi_{1} \in E_{1}, \dots, \Phi_{N} \in E_{N}$
\end{theorem}

\begin{proof}
Let $\eta \in Q$, $f_k$ a polynomial in $E_k$ and $N \in \Z_+$.  For notation we define $f_{\eta}^k(\theta_1,\dots,\theta_k) := f_{k}(\theta_1 + N \eta, \dots ,\theta_k + N^k \eta)$. Then for all $J \in \Z_+$, 

$\int_{Q}{\left\|\sum_{k=1}^J{f_{\eta}^k}\right\|_{L^p(Q^J)}^p \frac{d\eta}{(2\pi)^d}}$
\begin{align}
\label{Fubini1} 
&= \int_{Q}{\int_{Q}{\dots\int_{Q}{\left\|\sum_{k=1}^J{f_{k}(\theta_1 + N \eta, \dots ,\theta_k + N^k \eta)}\right\|_{\C^m}^p \frac{d\s_1}{(2\pi)^d} \dots \frac{d\s_J}{(2\pi)^d} \frac{d \eta}{(2\pi)^d}}}}\\ \nonumber
&=\int_{Q}{\int_{Q}{\dots\int_{Q}{\left\|\sum_{k=1}^J{f_{k}(u_1, ,\dots u_k)}\right\|_{\C^m}^p \frac{du_1}{(2\pi)^d} \dots \frac{du_J}{(2\pi)^d} \frac{d \eta}{(2\pi)^d}}}}\\ \nonumber
&=\left\|\sum_{k=1}^J{f_{k}(u_{1}, \dots, u_{k})}\right\|_{L^p(Q^J)}^p
\end{align}
Similarly, 
\begin{equation}
\label{Fubini2}
\int_{Q}{\left\|\sum_{k=1}^J{T_{\widetilde{M}}^k f_{\eta}^k(\s_1, \dots, \s_k)}\right\|_{L^{p_0}(Q^J)}^{p_0}  \frac{d\eta}{(2\pi)^d}} = \left\|\sum_{k=1}^J{T_{\widetilde{M}}^k f_{k}(u_{1}, \dots, u_{k})}\right\|_{L^{p_0}(Q^J)}^{p_0}
\end{equation}

For fixed $\vec{\s} = (\s_1,\dots,\s_k)$ we will use the notation $F_{\s}^k(\eta):=f_{\eta}^k(\s_1,\dots, \s_k)$ to emphasize that we are thinking of the quantity as a function of $\eta$ and $\vec{\theta}$ as a parameter.  This is because $T_{\widetilde{M}}$ can only act on a function of one variable $\eta \in Q.$  Also, from this point on in the proof, all measures will be normalized Lebesgue measures.  We claim that 
\begin{equation}
\label{unifconv}
\lim_{N \to \infty}\left|\int_{Q}{\left\|\sum_{k=1}^J{T_{\widetilde{M}}^kf_{\eta}^k(\s_1,\dots, \s_k)}\right\|_{L^{p_0}(Q^J,\vec{\s})}^{p_0}\!\!\!\!\!\!\!\!\!\!\!\!d\eta} - \int_{Q^J}{\left\|\sum_{k=1}^J{T_{\widetilde{M}}F_{\s}^k(\eta)}\right\|_{L^{p_0}(Q,\eta)}^{p_0}\!\!\!\!\!\!\!\!\!\!\!\!d\vec{\s}}\right| = 0
\end{equation}
with convergence being uniform.  Let us prove this claim.  

Since $f_{\eta}^k(\s_1,\dots, \s_k) = \sum_{l_1 \in \Z^n}{\dots \sum_{l_k \in \Z^n}{x_{l_1,\dots,l_k}e^{i(l_1,\s_1 + N\eta)}\dots e^{i(l_k,\s_k + N^k\eta)}}}$ then 
\beno
T_{\widetilde{M}}^kf_{\eta}^k(\s_1,\dots, \s_k) = \sum_{l_1 \in \Z^n}{\dots \sum_{l_k \in \Z^n}{M(l_k)x_{l_1,\dots,l_k}e^{i(l_1,\s_1 + N\eta)}\dots e^{i(l_k,\s_k + N^k\eta)}}}
\eeno
and $T_{\widetilde{M}}F_{\s}^k(\eta) = \sum_{l_1 \in \Z^n}{\dots \sum_{l_k \in \Z^n}{M(Nl_1 +\dots+N^k l_k)x_{l_1,\dots,l_k}e^{i(l_1,\s_1 + N\eta)}\dots e^{i(l_1,\s_1 + N^k\eta)}}}$
But, $M(Nl_1 +\dots+N^kl_k) = M(l_k + \frac{1}{N}l_{k-1}+\dots+\frac{1}{N^{k-1}}l_1) \to M(l_k)$ as $N \to \infty$. Note that $l_k \neq 0$ by how we defined $E_k,$ so the fact that $M$ may not be continuous at the origin is fine.  Then, 
\beno
\left(\int_{Q}{\left\|\sum_{k=1}^J{T_{\widetilde{M}}^kf_{\eta}^k(\s_1,\dots, \s_k)}\right\|_{L^{p_0}(Q^J,\vec{\s})}^{p_0}d\eta}\right)^{\frac{1}{p_0}} - \left(\int_{Q^J}{\left\|\sum_{k=1}^J{T_{\widetilde{M}}F_{\s}^k(\eta)}\right\|_{L^{p_0}(Q,\eta)}^{p_0}d\vec{\s}}\right)^{\frac{1}{p_0}}
\eeno
\begin{align*}
& = \left(\int_{Q}{\left\|\sum_{k=1}^J{T_{\widetilde{M}}^kf_{\eta}^k(\s_1,\dots, \s_k)}\right\|_{L^{p_0}(Q^J,\vec{\s})}^{p_0}d\eta}\right)^{\frac{1}{p_0}} - \left(\int_{Q}{\left\|\sum_{k=1}^J{T_{\widetilde{M}}F_{\s}^k(\eta)}\right\|_{L^{p_0}(Q^J,\vec{\s})}^{p_0}d\eta}\right)^{\frac{1}{p_0}}\\
& \leq \left(\int_Q{\left\|\sum_{k=1}^J{T_{\widetilde{M}}^kf_{\eta}^k(\s_1,\dots,\s_k)-T_{\widetilde{M}}F_{\s}^k(\eta)}\right\|_{L^{p_0}(Q^J,\vec{\s})}^{p_0}d\eta}\right)^{\frac{1}{p_0}}\\
& \leq \left(\sum_{k=1}^J{\sum_{l_1 \in \Z^n}{\dots \sum_{l_k \in \Z^n}{\|x_{l_1,\dots,l_k}\|_{\C^m}^{p_0}\|M(Nl_1+\dots+N^kl_k)-M(l_k)\|^{p_0}}}}\right)^{\frac{1}{p_0}} \xrightarrow[]{N \to \infty} 0
\end{align*}
This completes the proof of the claim.  

Now we just need some easy estimates to finish the proof of the Theorem.  For notation, let $\vec{u} = (u_1,\dots,u_J) \in Q^J$ and $du_i$ the normalized Lebesgue measure on $\T^d$.  Then for sufficiently large $N$ and a given $\e > 0$, we can use (\ref{Fubini1}), (\ref{Fubini2}) and (\ref{unifconv}) to get

$\left\|\sum_{k=1}^J{T_{\widetilde{M}}^k f_k(u_1,\dots,u_k)}\right\|_{L^{p_0}(Q^J,\vec{u})}^{p_0}$

\begin{align*}
&=\int_Q{\left\|\sum_{k=1}^J{T_{\widetilde{M}}^k f_{\eta}^k(\s_1,\dots,\s_k)}\right\|_{L^{p_0}(Q^J,\vec{u})}^{p_0}d\eta}\\ \nonumber
& \leq \int_{Q^J}{\left\|\sum_{k=1}^J {T_{\widetilde{M}} F_{\s}^k(\eta)}\right\|_{L^{p_0}(Q,\eta)}^{p_0} d\vec{\s}} + \e\\\nonumber
& \leq \|T_{\widetilde{M}}\|_{L^p(Q) \to L^{p_0}(Q)}^{p_0} \int_{Q^J}{\left\|\sum_{k=1}^J { F_{\s}^k(\eta)}\right\|_{L^{p}(Q,\eta)}^{p_0} d\vec{\s}} + \e\\ \nonumber
& = \|T_{\widetilde{M}}\|_{L^p(Q) \to L^{p_0}(Q)}^{p_0} \int_{Q^J}{\left(\int_{Q}{\left\|\sum_{k=1}^J { F_{\s}^k(\eta)}\right\|_{\C^m}^p d\eta}\right)^{\frac{p_0}{p}} d\vec{\s}} + \e
\end{align*}
\be \label{Jensen}
& \leq & \|T_{\widetilde{M}}\|_{L^p(Q) \to L^{p_0}(Q)}^{p_0} \left(\int_{Q^J}{\int_{Q}{\left\|\sum_{k=1}^J{F_{\theta}^k(\eta)}\right\|_{\C^m}^p d\eta d\vec{\theta}}}\right)^{\frac{p_0}{p}} + \e\\ \nonumber
& = & \|T_{\widetilde{M}}\|_{L^p(Q) \to L^{p_0}(Q)}^{p_0} \left(\int_{Q^J}{\left\|\sum_{k=1}^J{f_{\eta}^k(\theta_1, \dots, \theta_k)}\right\|_{L^p(Q,\eta)}^p d\vec{\theta}}\right)^{\frac{p_0}{p}} + \e\\ \nonumber
& = & \|T_{\widetilde{M}}\|_{L^p(Q) \to L^{p_0}(Q)}^{p_0}\left\|\sum_{k=1}^J{f_k(u_1,\dots, u_k)}\right\|_{L^p(Q^J, \vec{u})}^{p_0}+\e \nonumber
\ee

Note that we used the concave version of Jensen's inequality in (\ref{Jensen}).\qedhere
\end{proof}

\begin{theorem}
\label{operatorlowerbound}
Fix $\tau \in \R$ and let $M = \left(\begin{array}{c}
m\\
\tau \end{array} \right)$ such that $m:\R^d \to \R$ be even, continuous (except possibly at the origin) and homogeneous of order zero multiplier.  If there exists $v^{\pm} \in S^{d-1}$ (the unit sphere) such that $m(v^{\pm}) = \delta^{\pm},$ $\delta^+ = -\delta^- > 0$ then $\|T_{M}\|_{L^p(\R^d,\C) \to L^{p_0}(\R^d,\C^2)} \geq \frac{2 C_{p, p_0}^{\tau}}{\delta^+ - \delta^-}.$
\end{theorem}

\begin{proof}
For this proof we will continue using the notation $Q$ to represent $\T^d.$  Let $\e > 0$ be given.  By the definition of $C_{p,p_0}^{\tau}$, there exists $\vec{\beta} = (\beta_1, \dots,\beta_N)$ and complex-valued $d_1(r_0) ,d_2(r_0,r_1),\dots,d_N(r_0,\dots,r_{N-1}),$ where $\beta_k \in \{\pm1\}$ such that 
\begin{equation}
\label{vUMD}
C_{p,p_0}^{\tau}-\frac{\left\|\sum_{k=1}^N{\left(\begin{array}{c}
\beta_k\\
\tau \end{array}\right)r_k d_k(r_0,\dots,r_{k-1})}\right\|_{p_0}}{\|\sum_{k=1}^N{r_k d_k(r_0,\dots,r_{k-1})}\|_p}
 \leq \e
\end{equation}
Let $\delta > 0$ be given and $m(v^{\pm}) = \delta^{\pm},$ then there exists $\{n_k^{\pm}\} \subset \Z^d$ such that $\|m(n_k^{\pm}) - m(v^{\pm})\| < \delta$ for $k$ sufficiently large.  To ease notation, let $n^{\pm}$ be one such $n_k^{\pm}$.  Since sign is an odd $L^p$ function on $[-\pi, \pi)$ then $\sign(\theta) = \sum_{j \in \Z \setminus \{0\}}{c_j \sin(j\theta)}$ for some $c_j$'s, where $\theta \in [-\pi, \pi)$.
Let $\psi^{\pm}(\s) = \sign(n^{\pm},\s),\,\,\s \in Q$ and $\{\alpha_k\}_{k=0}^N$ an arbitrary sequence from $\{\delta^{\pm}\}$.  Then if we define 
\begin{displaymath}
   \psi_k := \left\{
     \begin{array}{lr}
       \psi^+ &  \alpha_k = \delta^+\\
       \psi^- & \alpha_k = \delta^-
     \end{array}
   \right.
\end{displaymath}
an easy computation shows that 
\begin{displaymath}
   T_{\widetilde{M}}^k[\psi_k(\s_k) d_k(\psi_0(\s_0),\dots,\psi_{k-1}(\s_{k-1}))] = \left\{
     \begin{array}{lr}
       M(n^+)\psi_k(\s_k) d_k(\psi_0(\s_0),\dots,\psi_{k-1}(\s_{k-1}))&,\, \alpha_k \in \delta^+\\
       M(n^-)\psi_k(\s_k) d_k(\psi_0(\s_0),\dots,\psi_{k-1}(\s_{k-1}))&,\, \alpha_k \in \delta^-
     \end{array}
   \right.
\end{displaymath}
 
Since $(r_0,\dots,r_{N})$ has the same distribution on $[0,1)$ as $(\psi_0(\s_1),\dots,\psi_{N}(\psi_{N}))$  has on $Q^N$ then

$\left\|\sum_{k=1}^N{\left(\begin{array}{c}
\al_k\\
\tau \end{array}\right)r_k d_k(r_0,\dots,r_{k-1})}\right\|_{L^{p_0}[0,1)}$
\begin{align*}
&= \left\|\sum_{k=1}^N{\left(\begin{array}{c}
\al_k\\
\tau \end{array}\right)\psi_k(\theta_k) d_k(\psi_0(\theta_0),\dots,\psi_{k-1}(\theta_{k-1}))}\right\|_{L^{p_0}(Q^N)}\\
& \leq \left\|\sum_{k=1}^N{M((n^+ \vee n^-)_k) \psi_k(\s_k) d_k(\psi_0(\s_0),\dots,\psi_{k-1}(\s_{k-1}))}\right\|_{L^{p_0}(Q^N)}\\
& \,\,\,\,\,\,\,\,\,\,\,\,\,\,\,\,\,\,\,\,\,\,\,\,+ \left\|\sum_{k=1}^N{\left[\left(\begin{array}{c}
\al_k\\
\tau \end{array}\right) - M((n^+ \vee n^-)_k)\right] \psi_k(\s_k) d_k(\psi_0(\s_0),\dots,\psi_{k-1}(\s_{k-1}))}\right\|_{L^{p_0}(Q^N)}\\
&=: I + II
\end{align*}
The notation \begin{displaymath}
   M((n^+ \vee n^-)_k) = \left\{
     \begin{array}{lr}
       M(n^+) &, \,\,\text{if}\,\,  \alpha_k = \delta^+\\
       M(n^-) &, \,\,\text{if}\,\,  \alpha_k = \delta^-
     \end{array}
   \right.
\end{displaymath}
and we similarly define $M((v^+ \vee v^-)_k).$  Now, 
\beno
II &=& \left\|\sum_{k=1}^N{[m((v^+ \vee v^-)_k) - m((n^+ \vee n^-)_k)]\psi_k(\theta_k)d_k(\psi_0(\theta_0),\dots, \psi_{k-1}(\theta_{k-1}))}\right\|_{L^{p_0}(Q^N)}\\
& \leq & \delta \sum_{k=1}^N{\|\psi_k(\theta_k)d_k(\psi_0(\theta_0),\dots, \psi_{k-1}(\theta_{k-1}))}\|_{L^{p_0}(Q^N)} \xrightarrow[]{\delta \to 0} 0
\eeno
Using Theorem \ref{bourgain} gives
\begin{align*}
I &= \left\|\sum_{k=1}^N{T_{\widetilde{M}}^k [\psi_k(\s_k) d_k(\psi_0(\s_0),\dots,\psi_{k-1}(\s_{k-1}))]}\right\|_{L^{p_0}(Q^N)}\\
& \leq \|T_{\widetilde{M}}\|_{L^p(Q,\C) \to L^{p_0}(Q,\C^2)}\left\|\sum_{k=1}^N{\psi_k(\s_k) d_k(\psi_0(\s_0),\dots,\psi_{k-1}(\s_{k-1}))}\right\|_{L^{p}(Q^N)}\\
&= \|T_{\widetilde{M}}\|_{L^p(Q,\C) \to L^{p_0}(Q,\C^2)} \left\|\sum_{k=1}^N{r_k d_k(r_0,\dots,r_{k-1})}\right\|_{L^{p}[0,1)}
\end{align*}

So, all together we have 

$\left\|\sum_{k=1}^N{\left(\begin{array}{c}
\alpha_k\\ \nonumber
\tau \end{array}\right) r_k d_k(r_0,\dots,r_{k-1})}\right\|_{L^{p_0}[0,1)}$
\begin{equation}
\label{mainest}
\leq \|T_{\widetilde{M}}\|_{L^p(Q,\C) \to L^{p_0}(Q,\C^2)} \left\|\sum_{k=1}^N{r_k d_k(r_0,\dots,r_{k-1})}\right\|_{L^{p}[0,1)}
\end{equation}

Now that we have our main estimate we just need to rescale $\delta^{\pm}$ to $\pm 1$ to get $C_{p, p_0}^{\tau}$ into the estimate.  Let $A = \frac{2}{\delta^+ - \delta^-},$ then we can choose the correct $\alpha_k$ so that 
\begin{displaymath}
   A \alpha_k =: \beta_k = \left\{
     \begin{array}{lr}
       +1 ,& \alpha_k = \delta^+\\
       -1 ,& \alpha_k = \delta^-,
     \end{array}
   \right.
\end{displaymath}
where $\beta_k$ satisfies the estimate in (\ref{vUMD}).  We can use (\ref{vUMD}) and (\ref{mainest}) to get
\beno 
C_{p,p_0}^{\tau} &\leq& \frac{\left\|\sum_{k=1}^N{\left(\begin{array}{c}
\beta_k\\ \nonumber
\tau \end{array}\right)r_k d_k(r_0,\dots,r_{k-1})}\right\|_{p_0}}{\|\sum_{k=1}^N{r_k d_k(r_0,\dots,r_{k-1})}\|_p} +  \e\\ \nonumber
& \leq &  |A|\,\, \|T_{\widetilde{M}}\|_{L^p(Q,\C) \to L^{p_0}(Q,\C^2)} + \e\\ \nonumber
& \leq & |A| \|T_{M}\|_{L^p(\R^d,\C) \to L^{p_0}(\R^d,\C^2)} + \e 
\eeno
where the last inequality is due to Theorem \ref{transference1}. \qedhere
\end{proof}

\begin{cor} \label{mainresultlowerbound}
Fix $\tau \in \R$ and let $M = \left(\begin{array}{c}
m\\
\tau \end{array} \right),$ where m is the multiplier corresponding to the real or imaginary part of the Ahlfors-Beurling transform.  Then $m$ takes a maximum and minimum value of $\pm 1$ on the unit circle and therefore $\|T_{M}\|_{L^p(\C,\C) \to L^{p}(\C,\C^2)} \geq C_{p,p}^\tau.$  
\end{cor}
Note that $C_{p,p}^\tau$ was computed as $((p^*-1)^2+\tau^2)^{\frac 12},$ for $|\tau| \leq \frac 12$ and $1<p<2$ or $\tau \in \R$ and $2 \leq p < \infty$ in \cite{BJV1}, \cite{BJV2} and \cite{BJV3}.  For a definition and properties of the Ahlfors-Beurling operator look ahead to Section \ref{matrixapplications}.  

\begin{remark}
We expect the converse inequality to be the same.  So this lower bound technique will be sharp just as when $\tau = 0.$  This tells us that if we are able to determine the operator norm of some perturbation of the martingale transform, then we will be able to determine the lower bound of the same perturbation of $\Re B$ and $\Im B,$ using the same proof as here.
\end{remark}

\section{\bf Matrix--Valued Multipliers}

\subsection{\bf Multipliers}
\label{}

Generalizing the idea of a scalar muliplier, let us consider $M = \{m_{ij}\}_{1 \leq i,j \leq m}$ as a matrix-valued function whose entries are bounded functions from $\R^n$ to $\C$ and $p,p_0$ such that $1 \leq p,p_0 \leq \infty$.  We say that $M$ is the multiplier corresponding to the operator $T_M$, if $T_Mf:= (M \widehat{f})^{\vee}$ is a bounded operator from $L^p(\R^n,\C^m)$ to $L^{p_0}(\R^n,\C^m)$, considering functions $f$ from $\R^n$ to $\C^m$ with components in $\mathcal{S}$.  Note that the Fourier transform of $f$ is defined componentwise and $\|f\|_{L^p(\R^n,\C^m)}^p := \int_{\R^n}{\|f(x)\|_{\C^m}^p \, dx}$.  We use the notation $\widetilde{M} = (\widetilde{m}_{ij})_{1 \leq i,j \leq m}$ to denote the discrete multiplier.

\begin{remark}
We will use the notation $\mathscr{M}_{m \times n}^{p,p_0}(\R^d)$ to denote the class of all $m \times n$ multipliers whose components are continuous (except possibly at zero) and homogeneous of order zero , where $\|M\|_{\mathscr{M}_{m \times n}^{p,p_0}} := \|T_M\|_{L^p(\R^d, \C^n) \to L^p(\R^d, \C^m)}$ for any such $M \in \mathscr{M}_{m \times n}^{p,p_0}(\R^d).$
\end{remark}

\subsection{\bf Generalizing Unconditional Martingale Differences}
\label{}

The idea of an Unconditional Martingale Difference ($UMD$) will be generalized here.  Let $X$ be a Banach space and $\{r_n\}_{n \geq 0}$ the Rademacher sequence.  Let $\{F_N\}_{N \geq 0}$ and $\{G_N\}_{N \geq 0}$ be $X$-valued martingale difference sequences of the form 
$$F_N = \sum_{k=1}^N{d_k(r_0,\dots,r_{k-1})r_k},\,\,\,\, G_N = \sum_{k=1}^N{\beta_k d_k(r_0,\dots,r_{k-1})r_k}$$
where $d_{k}:\{\pm 1\}^k \to X$ is Bochner integrable and $\vec{\beta}$ is a vector with entries $\beta_{i} \in {\pm 1}$.  For any particular such $\vec{\beta}$, $MT_{\vec{\beta}}(F_N) := G_N$ is a martingale transform.

\begin{defi}
\label{UMDpp0}
$UMD_{p,p_0}(X) := \sup_{\vec{\beta}} \|MT_{\vec{\beta}}\|_{L_{X}^p[0,1) \to L_{X}^{p_0}[0,1)}$
\end{defi}

Note that $UMD_p(X) = UMD_{p,p}(X)$. For definitions and properties regarding $UMD_p(X)$ for a Banach space $X$,  refer to \cite{Burkholder}.  There is one property that is very useful in getting sharp estimates for singular integrals.  The property is that $UMD_p(\mathscr{H}) = p^*-1$, where $\mathscr{H}$ is a Hilbert space.
\subsection{\bf Transference}
\label{}
 
The same proofs from Section \ref{transference} apply here to the following.  
 
\begin{lemma}
Suppose $1 < p_0 \leq p < \infty$ and $M \in \mathscr{M}_{m \times n}^{p,p_0}(\R^d),$ then $\|T_{M}\|_{L^{p}(\R^d,\C^n) \to L^{p_0}(\R^d,\C^m)} \geq \|T_{\widetilde{M}}\|_{L^{p}(\T^d,\C^n) \to L^{p_0}(\T^d,\C^m)}$.
\end{lemma}

\begin{lemma}
Suppose $1 < p \leq p_0 < \infty$ and $M \in \mathscr{M}_{m \times n}^{p,p_0}(\R^d),$ then $\|T_{M}\|_{L^{p}(\R^d,\C^n) \to L^{p_0}(\R^d,\C^m)} \leq \|T_{\widetilde{M}}\|_{L^{p}(\T^d,\C^n) \to L^{p_0}(\T^d,\C^m)}$.
\end{lemma}

\subsection{\bf Lower Bound of Matrix--valued Multipliers}

Let $\mathscr{P}_k$ be a trigonometric polynomial taking $k$ variables in $Q := \T^d$.  The following Theorems have nearly identical proofs to those in Section \ref{lowerbound} so we omit the details.

\begin{theorem}
Let $M \in \mathscr{M}_{m \times n}^{p,p_0}(\R^d),$ corresponding to $T_M$ and 
$$T_{\widetilde{M}}^{k}\left(\sum_{j_1 \in \Z^n} \dots \sum_{j_k \in \Z^n}{x_{j_1, \dots, j_k}e^{i (j_1,\theta_{1})} \dots e^{i (j_k,\theta_{k})}}\right) :=
\sum_{j_1 \in \Z^n} \dots \sum_{j_k \in \Z^n}{M(j_k)x_{j_1, \dots, j_k}e^{i (j_1,\theta_{1})} \dots e^{i (j_k,\theta_{k})}}$$
If $\Phi_{k} \in E_k = \clos_{L^p}\{\mathscr{P}_k: \int_{Q}{\mathscr{P}_k(\theta_{1}, \dots, \theta_{k}) d\theta_{k}}=0\}$ and $p \geq p_0 > 1,$ then 
$$\left\|\sum_{k=1}^{N}{T_{\widetilde{M}}^{k}\Phi_k(\theta_{1}, \dots, \theta_{k})}\right\|_{L^{p_0}(Q^d,\C^m)} \leq \|T_{\widetilde{M}}\|_{L^{p}(Q,\C^n) \to L^{p_0}(Q,\C^m)} \left\|\sum_{k=1}^{N}{\Phi_k(\theta_{1}, \dots, \theta_{k})}\right\|_{L^{p}(Q^d,\C^n)}$$
for all $\Phi_{1} \in E_{1}, \dots, \Phi_{N} \in E_{N}$
\end{theorem}

For $1 < p_0 \leq p < \infty$, denote $\mathscr{M}_{p,p_0}(\mathbb{M}_m)$ as the collection of $m \times m$ matrix valued multipliers whose components are real, even, continuous (except possibly at zero) and homogeneous order zero functions in $\R^n$ with $\|M\|_{\mathscr{M}_{p,p_0}(\R^n)} := \|T_M\|_{L^p \to L^p_0}$.  We will use the notation $\mathscr{M}_{p}(\mathbb{M}_m)$ for $\mathscr{M}_{p,p}(\mathbb{M}_m)$.

\begin{theorem}
\label{operatorlowerbound}
Suppose $M \in \mathscr{M}_{p,p_0}(\mathbb{M}_m)$.  If there exists $v^{\pm} \in \R^n$ such that $M(v^{\pm}) = \delta^{\pm} U$, where $U$ is a unitary matrix and $\delta^{+} > \delta^{-}$, then 
\begin{equation}
\|T_{M}\|_{L^{p}(\R^n,\C^m) \to L^{p_0}(\R^n,\C^m)} \geq \|T_{\widetilde{M}}\|_{L^{p}(\T^n,\C^m) \to L^{p_0}(\T^n,\C^m)} \geq \frac{UMD_{p,p_0}(\C^m)}{\frac{2}{\delta^{+} - \delta^{-}}\left(1+ \frac{|\delta^+ + \delta^-|}{|\delta^+| + |\delta^-|}\right)}. \nonumber
\end{equation}
\end{theorem}

\begin{remark}
{\rm The main application of Theorem \ref{operatorlowerbound} will be to matices $M$ that are unitarily diagonalizable and satisfy the needed assumptions.  Recall, for example, that self-adjoint matrices are unitarily diagonalizable.  Also, a $2 \times 2$ matrix that is the sum of a skew-symmetric matrix and diagonal matrix is unitarily diagonalizable as well, as we will use as an application in the next section.}  
\end{remark}

\begin{remark}
{\rm Notice how in Theorem \ref{operatorlowerbound} if we use a multiplier $M$ that is nearly the identity matrix, but still satisfying all of the assumptions, then the estimate is quite bad, even though $\|T_M\|$ should be approximately $1$.  This suggests that there is room for improvement with our result.}
\end{remark}

\subsection{\bf Applications}\label{matrixapplications}

For the Ahlfors-Beurling operator defined as
$$B f(z) := p.v. -\frac{1}{\pi} \int_{\C} \frac{f(w)}{(z-w)^2} \, dm_{2}(w),$$ 
it is known that $\|B\|_p := \|B\|_{L^p(\C) \to L^p(\C)} \geq p^* -1,$  for all $p$ such that $1 < p < \infty,$ as originally shown by Lehto in \cite{Lehto} in 1965.  There has not been (to the authors knowledge) an alternate proof of Lehto's result, but it now follows as an easy application of Theorem \ref{operatorlowerbound}. 

\medskip
\begin{cor}
\label{Beurlinglowerbd}
$\|B\|_p \geq p^* -1,$ for all $p$ such that $1 < p < \infty$
\end{cor}
\medskip

\begin{proof}
It is known that $-B = R_{2}^2 - R_{1}^2 + 2iR_{1}R_{2}$ ($R_1$ and $R_2$ are the planar Riesz transforms) and the multiplier corresponding to $B$ is $m(\xi) = \frac{\xi_{2}^2 - \xi_{1}^2 + 2i \xi_{1}\xi_{2}}{|\xi|^2} =: m_{R}(\xi) + i m_{I}(\xi).$  Let $M =\left(
  \begin{array}{ c c }
     m_{R}(\xi) & m_{I}(\xi) \\
     -m_{I}(\xi) & m_{R}(\xi)
  \end{array} \right)$ and $\widehat{f} = u + iv$.  Since we have that $m(\xi) \widehat{f} = M(\xi)\left(
  \begin{array} {c}
  u\\
  v
  \end{array} \right)$, then there exists an isomorphism taking $(m\widehat{f})^{\vee}$ to $(M\widehat{f})^{\vee}$ and preserving the norm.  So we can conclude that $\|B\|_p = \|T_M\|_{L^p(\C) \to L^p(\R^2)}$.  Observe that $m_{R}(1,0) = 1$ and $m_{R}(0,1) = -1$ are the maximum and minumum of $m_R$ on $\partial B(0,1)$ respectively.  Theorem \ref{operatorlowerbound} implies that $\|T_M\|_{L^p(\C) \to L^p(\R^2)} \geq p^* -1$.
\end{proof}

\begin{cor}
\label{QuantumCombOp}
$\|(R_{1}^2-R_{2}^2)\cos(\theta) + 2 R_1 R_2 \sin(\theta)\|_{L^p(\C) \to L^p(\C)} = p^* - 1$ for all $\theta$ and all $1 < p < \infty$
\end{cor}
\medskip 
  
\begin{proof}
If $|\xi|=1$ then the multiplier corresponding to the operator is given by $(\xi_{1}^2 - \xi_{2}^2)\cos\theta + 2 \xi_1 \xi_2 \sin\theta$ and can be written as $\cos(2\phi - \theta)$ by letting $\xi_1 = \cos\phi$ and $\xi_2 = \sin\phi$.  So the multipier will achive its maximum and minimum of $1$ and $-1$ on the unit circle.  By Theorem \ref{operatorlowerbound} we have the lower bound of the operator is $p^* -1$.  The upper bound comes from \cite{Volberg}.
\end{proof}
 
\medskip

\begin{assumption}
For the remainder of the paper let us assume that the Iwaniec conjecture ($\|B\|_p = p^* -1$ for all $1<p<\infty$) is true.  
\end{assumption}

\begin{cor}
\label{PertBeurling}
\begin{displaymath}
   \|c(R_{1}^2-R_{2}^2) + 2 iR_1 R_2 \|_{L^p(\C) \to L^p(\C)} = \left\{
     \begin{array}{lr}
       (p^*-1) &, \,0 < |c| < 1\\
       |c|(p^*-1) &, \, |c| \geq 1
     \end{array}
   \right.
\end{displaymath}
for all $p$ such that $1<p< \infty.$

\end{cor}
\medskip

\begin{proof}
If $c \geq 1$, then it is easy to see that the operater is bounded below by $c(p^*-1)$, by simply using the same argument as in Corollary \ref{Beurlinglowerbd}.  The upper bound for the operator is $(c-1)\|R_1^2-R_2^2\|_p+\|B\|_p \leq c(p^*-1)$ by \cite{Nazarov} and since $\|B\|_p = p^* -1.$

\smallskip
If $0 < c < 1$, then we can write the operator as $c \|(R_{1}^2-R_{2}^2) + 2 i \frac{1}{c} R_1 R_2 \|_{p} = c \| \frac{1}{c} (R_{1}^2-R_{2}^2) + 2 i R_1 R_2 \|_{p}$, since $(R_{1}^2-R_{2}^2)$ is unitarily equivalent to $2 iR_1 R_2$ (using a rotation by $\frac{\pi}{4}$). Applying the case when $c \geq 1$ gives the desired result. 

\smallskip
For the cases where $c$ is negative we just need to observe that for multipliers of the form $m(\xi) = \frac{\left(\xi, A \xi \right)}{|\xi|^2}$, we have the property $(T_m f)(x) = \overline{(T_{\overline{m}}\overline{f})(x)}$.  This yields $\|T_m\|_p = \|T_{\overline{m}}\|_p$.  But, $\frac{c(\xi_{1}^2 - \xi_{2}^2) + 2 i  \xi_1 \xi_2}{|\xi|^2}$ is the multiplier corresponding the given operator and the conjugate corresponds to the operator $c(R_{1}^2-R_{2}^2) - 2 iR_1 R_2$.  So $\|c(R_{1}^2-R_{2}^2) + 2 iR_1 R_2\|_p = \|-c(R_{1}^2-R_{2}^2) - 2 iR_1 R_2\|_p = \|-c(R_{1}^2-R_{2}^2) + 2 iR_1 R_2\|_p$.  Using the two cases that we already have shown will give us the result for $c$ negative.
\end{proof}

\medskip  
\begin{cor}
\label{analytic}
Let $F(z) := (R_{1}^2-R_{2}^2) + 2 z R_1 R_2,$ then for all $p$ such that $1 < p < \infty$,  
 
 \begin{displaymath}
   \|F(z)\|_{L^p(\C) \to L^p(\C)} = \left\{
     \begin{array}{lr}
       \sqrt{1+z^2}(p^*-1), \,\, z \in \R\\
       (p^*-1), \,\, Re(z)=0, |z|<1\\
       |z|(p^*-1), \,\, Re(z)=0, |z| \geq 1
     \end{array}
   \right.
\end{displaymath}

\end{cor}
\medskip

\begin{proof}
If $z = x \in \R$ then $\|F(z)\|_p = \sqrt{1+x^2}\left\|\frac{1}{\sqrt{1+x^2}}(R_{1}^2-R_{2}^2) + \frac{x}{\sqrt{1+x^2}} 2 R_1 R_2 \right\|_{p} = \sqrt{1+x^2} (p^*-1)$ by making the substitution $x = \tan\theta$ and using Corollary \ref{QuantumCombOp}.  The other two cases follow directly from Corollary \ref{PertBeurling}. 
\end{proof} 

\bibliographystyle{amsplain}

\end{document}